\documentclass[12pt]{article}
\usepackage[utf8]{inputenc}
\usepackage[T1]{fontenc}
\usepackage{lmodern}
\usepackage{microtype}
\usepackage[margin=1in]{geometry}
\usepackage{amsmath}
\usepackage{amssymb}
\usepackage{amsthm}
\usepackage{mathtools}
\usepackage[english]{babel}
\usepackage{fancyhdr}
\usepackage{titling}
\usepackage[colorlinks=true,allcolors=blue]{hyperref}

\allowdisplaybreaks[1]
\AtBeginDocument{
  \setlength{\abovedisplayskip}{0.75em plus 0.25em minus 0.25em}
  \setlength{\belowdisplayskip}{0.75em plus 0.25em minus 0.25em}
  \setlength{\belowdisplayshortskip}{0.5em plus 0.25em minus 0.25em}
}

\title{Hitchin's Conjecture for $\mathfrak{sl}_n$ in Odd Prime Degree}
\author{Boming Jia}
\date{}
\posttitle{\par\end{center}\vskip0.25em}
\hypersetup{
  pdftitle={Hitchin's Conjecture for sl\_n in Odd Prime Degree},
  pdfauthor={Boming Jia},
  pdfsubject={Hitchin's conjecture for sl\_n in odd prime degree}
}

\newtheorem{theorem}{Theorem}[section]
\newtheorem{proposition}[theorem]{Proposition}
\newtheorem{lemma}[theorem]{Lemma}
\newtheorem{corollary}[theorem]{Corollary}
\theoremstyle{definition}
\newtheorem{definition}[theorem]{Definition}
\newtheorem{example}[theorem]{Example}
\theoremstyle{remark}
\newtheorem{remark}[theorem]{Remark}

\DeclareMathOperator{\St}{St}
\DeclareMathOperator{\ad}{ad}
\DeclareMathOperator{\tr}{tr}
\DeclareMathOperator{\sgn}{sgn}
\DeclareMathOperator{\End}{End}
\newcommand{\C}{\mathbb C}
\newcommand{\Z}{\mathbb Z}
\newcommand{\Fp}{\mathbb F_p}

\begin{document}
\setlength{\droptitle}{-6.75em}
\maketitle
\vspace{-4em}

\begin{abstract}
We prove Hitchin's conjecture for $\mathfrak{sl}_n$ in every odd prime
degree. More precisely, let $p=2k+1$ be prime and let $n\geq k+1$.
Every nonzero primitive element
$\alpha\in\Lambda^p(\mathfrak{sl}_n^*)^{\mathfrak{sl}_n}$
is nonzero on the top exterior power of the unique $p$-dimensional
irreducible submodule of $\mathfrak{sl}_n$ for the adjoint action
of a principal $\mathfrak{sl}_2$-subalgebra.
The conjecture also holds for types
$B_\ell$ and $C_\ell$, with $\ell\geq2$, in every prime degree
$p=4r-1$ with $1\leq r\leq\ell$.
\end{abstract}

\section{Introduction.}\label{sec:introduction}

Let $\mathfrak g$ be a complex simple Lie algebra. We identify
$H^*(\mathfrak g,\C)$ with the algebra
$\Lambda(\mathfrak g^*)^{\mathfrak g}$ of invariant alternating forms.
An element $\alpha$ is \emph{primitive} if its coproduct
in the standard Hopf algebra structure is
$\alpha\otimes1+1\otimes\alpha$
\cite[Theorem 9.2 and Section 10]{Koszul1950}.
An $\mathfrak{sl}_2$-subalgebra is \emph{principal} if it contains a
regular nilpotent element of $\mathfrak g$. Hitchin conjectured that
every nonzero primitive $d$-form is nonzero on the top exterior power
of some $d$-dimensional irreducible submodule for the adjoint action
of a principal $\mathfrak{sl}_2$-subalgebra
\cite[Section 2.2, final remark]{Hitchin2011}.

Let $G$ be a connected complex simple group with Lie algebra
$\mathfrak g$, and let $\mathcal M$ be the smooth locus of the moduli
space of stable principal $G$-bundles on a smooth projective complex
curve $C$. A nonzero invariant primitive $d$-form defines a map
\[
 H^1(C,K_C^{1-d})\longrightarrow H^0(\mathcal M,\Lambda^dT_{\mathcal M}),
\]
where $K_C$ is the canonical bundle. Hitchin proposed using the
conjecture to prove injectivity when $C$ has genus greater than two
\cite[Sections 2.1--2.2]{Hitchin2011}.

Bushek and Kumar reduced the conjecture to the simply-laced types
\cite[Theorem 2.5]{BushekKumar2014}. Hitchin proved it in the largest
primitive degree $2n-1$ for $\mathfrak{sl}_n$
\cite[Section 4.3]{Hitchin2023}. We prove the conjecture for
$\mathfrak{sl}_n$ in every odd prime degree.

\begin{theorem}\label{thm:main}
Let $p=2k+1$ be an odd prime and let $n\geq k+1$.
Fix a principal $\mathfrak{sl}_2$-subalgebra of $\mathfrak{sl}_n$, and
let $W_k\subset\mathfrak{sl}_n$ be the unique $p$-dimensional
irreducible submodule for its adjoint action. For every nonzero
primitive element
$\alpha\in\Lambda^p(\mathfrak{sl}_n^*)^{\mathfrak{sl}_n}$, we have
\[
 \left.\alpha\right|_{\Lambda^pW_k}\neq0.
\]
\end{theorem}

We make an explicit choice of basis of $W_k$ such that the primitive
form $\omega_p$ of Definition~\ref{def:alternating-trace} evaluates to
an integer coefficient $c_k(n)$ times an explicit nonzero factor.
In Lemma~\ref{lem:coefficient-congruence}, we prove the uniform congruence
$c_k(n)\equiv(-1)^{k(k+1)/2}\pmod p$ for every odd prime $p=2k+1$
and every $n\geq k+1$. Thus $c_k(n)$ is a nonzero integer, and
$\omega_p$ is nonzero on $\Lambda^pW_k$.

\vspace{1em}
\noindent\textbf{Acknowledgments.}
The author was supported by NSFC Grant No.~12225108 and the Shuimu Scholar
Program in Tsinghua University. The author used ChatGPT extensively during
the preparation of this work. He would like to thank the model GPT-5.6 Sol
for generating an earlier draft of this paper and collaborating with the
author throughout the revision process.

\setlength{\abovedisplayskip}{0.5em plus 0.25em minus 0.25em}
\setlength{\belowdisplayskip}{0.5em plus 0.25em minus 0.25em}
\setlength{\belowdisplayshortskip}{0.25em plus 0.25em minus 0.25em}

\section{Alternating products for a principal \texorpdfstring{$\mathfrak{sl}_2$}{sl2}.}\label{sec:symmetric}

Fix $n\geq2$. On $V_n=\C^n$ with basis $v_1,\ldots,v_n$, define
$e,f,h$ by
\begin{equation}\label{eq:basis-action}
 ev_j=v_{j-1},\qquad fv_j=j(n-j)v_{j+1},\qquad
 hv_j=(n+1-2j)v_j,
\end{equation}
where $v_0=v_{n+1}=0$. These matrices form a principal
$\mathfrak{sl}_2$-triple, since they satisfy the $\mathfrak{sl}_2$ relations
and $e$ is a single nilpotent Jordan block. For each $n$, we keep
this triple and basis throughout.

Write $(\ad f)(Y)=[f,Y]$. Conjugation by $\exp(tf)$ acts as
$\exp(t\ad f)$ and preserves products. By $[f,e]=-h$ and
$[f,h]=2f$, we have, for $1\leq k<n$,
\begin{equation}\label{eq:adjoint-orbit}
 \exp(t\ad f)(e^k)
 =\exp(tf)e^k\exp(-tf)=(e-th-t^2f)^k.
\end{equation}
\begin{definition}[Taylor coefficients]\label{def:taylor-coefficients}
For $1\leq k<n$, let
\begin{equation}\label{eq:basis}
 u_{i,k}=\frac{1}{i!}(\ad f)^i(e^k),\qquad 0\leq i\leq2k.
\end{equation}
These are the Taylor coefficients of \eqref{eq:adjoint-orbit}.
We denote their span in $\mathfrak{sl}_n$ by $W_k$.
\end{definition}

\begin{samepage}
\begin{example}[$\mathfrak{sl}_3$]\label{ex:sl3}
The coefficients of $\exp(t\ad f)(e^k)$ for $k=1$ and $k=2$ are,
respectively,
\begingroup
\setlength{\arraycolsep}{0.25em}
\[
\begin{bmatrix*}[r]
 0&1&0\\0&0&1\\0&0&0
\end{bmatrix*},
\begin{bmatrix*}[r]
 -2&0&0\\0&0&0\\0&0&2
\end{bmatrix*},
\begin{bmatrix*}[r]
 0&0&0\\-2&0&0\\0&-2&0
\end{bmatrix*}
\]
\[
\begin{bmatrix*}[r]
 0&0&1\\0&0&0\\0&0&0
\end{bmatrix*},
\begin{bmatrix*}[r]
 0&-2&0\\0&0&2\\0&0&0
\end{bmatrix*},
\begin{bmatrix*}[r]
 2&0&0\\0&-4&0\\0&0&2
\end{bmatrix*},
\begin{bmatrix*}[r]
 0&0&0\\4&0&0\\0&-4&0
\end{bmatrix*},
\begin{bmatrix*}[r]
 0&0&0\\0&0&0\\4&0&0
\end{bmatrix*}
\]
\endgroup
\end{example}
\end{samepage}

\begin{proposition}\label{prop:principal-module}
The matrices $u_{0,k},\ldots,u_{2k,k}$ have integer entries and form
a basis of the unique $(2k+1)$-dimensional irreducible submodule
$W_k\subset\mathfrak{sl}_n$ for the adjoint action of the fixed
principal triple.
\end{proposition}

\begin{proof}
By \eqref{eq:adjoint-orbit}, every $u_{i,k}$ has integer entries, and
\[
 u_{2k,k}=(-1)^kf^k.
\]

By induction using the $\mathfrak{sl}_2$-relations, we have
\begin{equation}\label{eq:string}
 [h,u_{i,k}]=2(k-i)u_{i,k},\qquad
 [e,u_{i,k}]=(2k-i+1)u_{i-1,k},
\end{equation}
where $u_{-1,k}=0$. Since $u_{0,k}=e^k\neq0$, all the $u_{i,k}$ are
nonzero and have distinct $\ad h$-weights, so they form a basis of $W_k$.
Since $[f,u_{i,k}]=(i+1)u_{i+1,k}$, where $u_{2k+1,k}=0$,
the subspace $W_k$ is invariant under the principal triple.
Every nonzero submodule of $W_k$ contains an $\ad h$-eigenvector,
hence some $u_{i,k}$ and, by repeated commutation with $e$ and $f$,
all of them. Thus $W_k$ is irreducible.

For uniqueness, let $A\in\End(V_n)$ satisfy $[e,A]=0$ and
$[h,A]=2kA$. Since $hv_n=(1-n)v_n$, we have
\[
 h(Av_n)=[h,A]v_n+A(hv_n)=(1-n+2k)Av_n.
\]
By \eqref{eq:basis-action}, we have $Av_n\in\C v_{n-k}$.
Write $Av_n=a e^kv_n$ for some $a\in\C$. Since $[e,A]=0$, we have
\[
 Ae^rv_n=e^rAv_n=a e^{r+k}v_n\qquad(0\leq r<n).
\]
Since $v_n,ev_n,\ldots,e^{n-1}v_n$ form a basis, we have $A=ae^k$.
Thus $\C e^k$ is the unique highest weight line of weight $2k$.
Every irreducible submodule of dimension $2k+1$ is generated by
this line and equals $W_k$.
\end{proof}

Fix an odd integer $p=2k+1$ with $1\leq k<n$, and write $u_i=u_{i,k}$.

\begin{definition}[Alternating trace]\label{def:alternating-trace}
For $n\times n$ matrices $Y_1,\ldots,Y_r$ over a field, set
\[
 \St_r(Y_1,\ldots,Y_r)
 =\sum_{\sigma\in S_r}\sgn(\sigma)Y_{\sigma(1)}\cdots Y_{\sigma(r)}.
\]
We define the $p$-linear form $\omega_p$ by
\[
 \omega_p(Y_0,\ldots,Y_{p-1})
 =\tr\bigl(\St_{p-1}(Y_0,\ldots,Y_{p-2})Y_{p-1}\bigr).
\]
\end{definition}

\begin{proposition}\label{prop:primitive-form}
The form $\omega_p$ is alternating and invariant under simultaneous
conjugation over every field. Over $\C$, its restriction to
$\mathfrak{sl}_n$ spans the space of primitive forms of degree $p$.
\end{proposition}

\begin{proof}
Since $p$ is odd, cyclic rotation preserves both sign and trace.
Grouping the terms in the full alternating sum into sets of $p$
cyclic rotations, we have
\[
 \tr\St_p(Y_0,\ldots,Y_{p-1})
 =p\,\omega_p(Y_0,\ldots,Y_{p-1}).
\]
Thus $\omega_p$ is alternating over $\C$, hence over every field
since its alternation identities are polynomial identities over $\Z$.
Since the trace is invariant under conjugation, $\omega_p$ is
invariant under simultaneous conjugation. Over $\C$, the restriction
of $\omega_p$ to $\mathfrak{sl}_n$ is a primitive generator
\cite[p.~141]{Dynkin1954}, \cite[Section 5.5, p.~261]{Kostant1958}.
\end{proof}

\begin{lemma}[The coefficient $c_k(n)$]\label{lem:coefficient}
There is a unique integer $c_k(n)$ such that
\begin{equation}\label{eq:coefficient-identity}
 \St_{p-1}(u_0,\ldots,u_{p-2})=c_k(n)e^k.
\end{equation}
\end{lemma}

\begin{proof}
Let
\[
 A=\St_{p-1}(u_0,\ldots,u_{p-2}).
\]
By the Leibniz rule for commutators and \eqref{eq:string}, we have
\[
 [e,A]=\sum_{i=0}^{p-2}
 \St_{p-1}(u_0,\ldots,[e,u_i],\ldots,u_{p-2})=0.
\]
Indeed, $[e,u_i]=(p-i)u_{i-1}$, so each summand has either the zero
argument $u_{-1}$ or two copies of $u_{i-1}$.

By the Leibniz rule for $\ad h$ on $\End(V_n)$, we have
\begin{align*}
 [h,A]
 &=\sum_{i=0}^{p-2}
   \St_{p-1}(u_0,\ldots,[h,u_i],\ldots,u_{p-2})\\
 &=2\left(\sum_{i=0}^{2k-1}(k-i)\right)A
 =2\bigl(2k^2-k(2k-1)\bigr)A=2kA.
\end{align*}
By the highest weight argument in Proposition~\ref{prop:principal-module},
we have $A=ae^k$ for some $a\in\C$.
The $(n-k,n)$-entry of $e^k$ is $1$, so $a$ is the corresponding
entry of $A$ and is an integer. Thus $c_k(n)=a$ is uniquely determined.
\end{proof}

\newpage
Let $c_k(n)$ be the integer in Lemma~\ref{lem:coefficient}.

\begin{example}\label{ex:coefficients}
\begin{align*}
 c_1(n)&=2,\\
 c_2(n)&=2^5\cdot3(44-5n^2),\\
 c_3(n)&=2^5\cdot3^6\cdot5\bigl(7(n^8-299n^6+18727n^4-420945n^2)+21701196\bigr).
\end{align*}
\end{example}

\begin{lemma}\label{lem:trace-periodicity}
For every $n\geq k+1$,
\begin{equation}\label{eq:trace-reduction}
 \omega_p(u_0,\ldots,u_{p-1})
 =(-1)^kc_k(n)(k!)^2\binom{n+k}{p}
\end{equation}
and
\begin{equation}\label{eq:periodicity}
 c_k(n+p)\equiv c_k(n)\pmod p.
\end{equation}
\end{lemma}

\begin{proof}
By Lemma~\ref{lem:coefficient} and $u_{p-1}=(-1)^kf^k$, we have
\[
 \omega_p(u_0,\ldots,u_{p-1})
 =c_k(n)\tr(e^ku_{p-1})=(-1)^kc_k(n)\tr(e^kf^k).
\]
By \eqref{eq:basis-action}, the only nonzero diagonal entries of
$e^kf^k$ have indices $1\leq j\leq n-k$. Thus
\begin{equation}\label{eq:binomial-trace}
\begin{aligned}
 \tr(e^kf^k)
 &=\sum_{j=1}^{n-k}\prod_{r=0}^{k-1}(j+r)(n-j-r)\\
 &=(k!)^2\sum_{j=1}^{n-k}\binom{j+k-1}{k}\binom{n-j}{k}
 =(k!)^2\binom{n+k}{p}.
\end{aligned}
\end{equation}
For the last equality, we count $p$-element subsets of
$\{1,\ldots,n+k\}$ by their middle element $j+k$, choosing $k$
elements below it and $k$ above it. This proves
\eqref{eq:trace-reduction}.

For \eqref{eq:periodicity}, we work over $\Z/p\Z$. In dimension $n+p$,
the span of $v_1,\ldots,v_n$ is preserved by $e,f,h$, since
$fv_n=np\,v_{n+1}=0$. Their restrictions are the operators
\eqref{eq:basis-action} in dimension $n$. By taking coefficients in
\eqref{eq:adjoint-orbit}, we see that the $u_i$ restrict to their
$n$-dimensional counterparts as well. Taking the $(1,k+1)$-entry
of the restricted identity \eqref{eq:coefficient-identity}, we obtain
\[
 c_k(n+p)
 =\bigl(\St_{p-1}(u_0,\ldots,u_{p-2})\bigr)_{1,k+1}
 =c_k(n)\quad\text{in }\Z/p\Z.\qedhere
\]
\end{proof}

\begin{remark}\label{rem:composite-periodicity}
The periodicity in \eqref{eq:periodicity} holds for every odd integer
$p=2k+1$. In Section~\ref{sec:finite-field}, we prove that $c_k(n)$
has nonzero residue modulo $p$ whenever $p$ is prime.
\end{remark}

\section{Alternating traces modulo \texorpdfstring{$p$}{p}.}\label{sec:finite-field}

From now on, assume that $p$ is prime. By \eqref{eq:basis-action}, we have
\[
 \exp(tf)v_j=\sum_{r=0}^{n-j}r!\binom{j+r-1}{r}\binom{n-j}{r}
 t^rv_{j+r}.
\]
Thus $\exp(tf)$ has entries in $\Z[t]$. Since $e^kf^r$ has zero
diagonal for $r\neq k$, we have by \eqref{eq:binomial-trace}
\begin{equation}\label{eq:exponential-trace}
 \tr(e^k\exp(tf))=\frac{t^k}{k!}\tr(e^kf^k)
 =k!\binom{n+k}{p}t^k.
\end{equation}

We have $\exp(sf)\exp(tf)=\exp((s+t)f)$ over $\Z[s,t]$
and hence after reduction.

\begin{lemma}[Vandermonde identities]\label{lem:vandermonde}
In $\Fp$, we have
\[
 \det(a^i)_{0\leq a,i<p}=(-1)^{k(k+1)/2}k!.
\]
For every $a\in\Fp^\times$, the permutation $i\mapsto ai$ of $\Fp$
has sign $a^k$.
\end{lemma}

\begin{proof}
For $1\leq r\leq k$, we have by Wilson's identity
\[
 -1=(p-1)!=(p-r)!\prod_{j=1}^{r-1}(p-j)
 =(-1)^{r-1}(p-r)!(r-1)!.
\]
Thus $r!(p-r)!=(-1)^r r$. By the Vandermonde formula, we have
\[
 \det(a^i)_{0\leq a,i<p}
 =\prod_{0\leq x<y<p}(y-x)
 =\prod_{r=1}^{p-1}r!
 =\prod_{r=1}^{k}r!(p-r)!
 =(-1)^{k(k+1)/2}k!.
\]

We permute the rows of $(i^j)_{0\leq i,j<p}$ by $i\mapsto ai$
and factor $a^j$ from column $j$. Since $a^p=a$, we obtain
\[
 \sgn(i\mapsto ai)=\prod_{j=0}^{p-1}a^j
 =a^{p(p-1)/2}=a^k.\qedhere
\]
\end{proof}

\begin{lemma}\label{lem:coefficient-congruence}
Let $p=2k+1$ be an odd prime. For every $n\geq k+1$,
\begin{equation}\label{eq:coefficient-congruence}
 c_k(n)\equiv(-1)^{k(k+1)/2}\pmod p,
\end{equation}
where $c_k(n)$ is defined as in Lemma~\ref{lem:coefficient}.
\end{lemma}

\begin{proof}
By \eqref{eq:periodicity}, it suffices to consider $k+1\leq n\leq k+p$.
We work over the finite field $\Fp$ and retain the same notation
for the reductions of matrices and polynomials.
Since $0\leq n-k-1<p$, we have
\[
 \binom{n+k}{p}
 =\prod_{j=1}^{n-k-1}\frac{p+j}{j}=1.
\]
All denominators are nonzero, and the empty product is $1$.
By \eqref{eq:trace-reduction} and Wilson's identity
$(-1)^k(k!)^2=-1$, we have
\[
 \omega_p(u_0,\ldots,u_{p-1})=-c_k(n).
\]

Let
\[
 X(t)=\exp(tf)e^k\exp(-tf)=\sum_{i=0}^{p-1}t^iu_i.
\]
By multilinearity and alternation of $\omega_p$, and then
Lemma~\ref{lem:vandermonde}, we have
\begin{equation}\label{eq:finite-field-evaluation}
\begin{aligned}
 \omega_p(X(0),\ldots,X(p-1))
 &=\left(\sum_{\sigma\in S_p}\sgn(\sigma)
 \prod_{a=0}^{p-1}a^{\sigma(a)}\right)
 \omega_p(u_0,\ldots,u_{p-1})\\
 &=\det(a^i)_{0\leq a,i<p}\,
 \omega_p(u_0,\ldots,u_{p-1})
 =(-1)^{k(k+1)/2+1}k!c_k(n).
\end{aligned}
\end{equation}

\begin{samepage}
We consider all permutations of the list
\[
 (0,1,2,\ldots,p-1).
\]
Two listings are equivalent if one is a cyclic rotation of the other.
We denote the class of $(a_0,\ldots,a_{p-1})$ by
$[a_0,\ldots,a_{p-1}]$. Each class has a unique representative with
$p-1$ last. We therefore index the sum defining
$\omega_p(X(0),\ldots,X(p-1))$ by these classes.
Since $p$ is odd, cyclic
rotation preserves the sign of $i\mapsto a_i$ and the trace
$\tr(X(a_0)\cdots X(a_{p-1}))$. Both are therefore well-defined on classes.

The additive group $\Fp$ acts on these classes by
\[
 b\cdot[a_0,\ldots,a_{p-1}]=[a_0+b,\ldots,a_{p-1}+b].
\]
Translation commutes with cyclic rotation, so this action is well-defined.
The orbit of $[a_0,\ldots,a_{p-1}]$ is
\[
 \bigl\{[a_0+b,\ldots,a_{p-1}+b]\mid b\in\Fp\bigr\}.
\]
By the orbit--stabilizer formula, its cardinality divides $p$,
so it is $1$ or $p$.
\end{samepage}

Sign and trace are constant on each orbit. Indeed, every nonzero
translation is an even $p$-cycle, and
\[
 X(t+b)=\exp(bf)X(t)\exp(-bf).
\]
If the orbit of $[a_0,\ldots,a_{p-1}]$ has size $p$, its $p$
translated classes are distinct.
After factoring out their common sign, its contribution is
\[
 \sum_{b\in\Fp}\tr\!\left(\prod_{i=0}^{p-1}X(a_i+b)\right)
 =p\,\tr\!\left(\prod_{i=0}^{p-1}X(a_i)\right)=0.
\]
Thus only classes fixed by every translation contribute.

For a fixed class, choose its representative beginning with $0,a$.
Since translation by $b$ fixes the class, $b+a$ follows $b$.
Thus every fixed class has the form $[0,a,\ldots,(p-1)a]$ with
$a\in\Fp^\times$.
Conversely, translation rotates each such representative and thus fixes
its class.

We have
\begin{align*}
 \tr\bigl(X(0)X(a)\cdots X((p-1)a)\bigr)
 &=\tr\!\left(\prod_{j=0}^{p-1}
   \bigl(\exp(jaf)e^k\exp(-jaf)\bigr)\right)\\
 &=\tr\bigl((e^k\exp(af))^{p-1}e^k\exp(-(p-1)af)\bigr)\\
 &=\tr\bigl((e^k\exp(af))^p\bigr)
 =\bigl(\tr(e^k\exp(af))\bigr)^p\\
 &=\tr(e^k\exp(af))=a^k k!.
\end{align*}
We use \eqref{eq:exponential-trace} in the last equality.

By Lemma~\ref{lem:vandermonde}, the class $[0,a,\ldots,(p-1)a]$ has
sign $a^k$. Since $a^{2k}=1$, we obtain
\[
 \omega_p(X(0),\ldots,X(p-1))
 =\sum_{a\in\Fp^\times}a^{2k}k!=-k!.
\]
Comparing with \eqref{eq:finite-field-evaluation} and cancelling $k!$,
we obtain $c_k(n)=(-1)^{k(k+1)/2}$.
\end{proof}

\begin{proof}[Proof of Theorem~\ref{thm:main}]
We return to the matrices over $\C$ from Section~\ref{sec:symmetric}.
By Lemma~\ref{lem:coefficient-congruence}, the integer $c_k(n)$ is nonzero.
Since $n+k\geq p$, every factor on the right of
\eqref{eq:trace-reduction} is nonzero.
Thus $\omega_p(u_0,\ldots,u_{p-1})\neq0$.
By Proposition~\ref{prop:primitive-form}, the form $\alpha$ is a nonzero
scalar multiple of $\omega_p$, so $\alpha(u_0,\ldots,u_{p-1})\neq0$.

Now let $\mathfrak s\subset\mathfrak{sl}_n$ be any principal
$\mathfrak{sl}_2$-subalgebra. By Kostant's conjugacy theorem
\cite[Corollary 3.7]{Kostant1959}, there exists $g\in\mathrm{SL}_n(\C)$
such that $\mathfrak s=g\langle e,h,f\rangle g^{-1}$.
Conjugation by $g$ intertwines the adjoint actions of these two
subalgebras. By Proposition~\ref{prop:principal-module}, the subspace
$gW_kg^{-1}$ is therefore the unique $p$-dimensional irreducible
submodule for $\mathfrak s$, with basis
$gu_0g^{-1},\ldots,gu_{p-1}g^{-1}$.
Since $\mathrm{SL}_n(\C)$ is connected, the form $\alpha$ is invariant
under its adjoint action. Hence
\[
 \alpha(gu_0g^{-1},\ldots,gu_{p-1}g^{-1})
 =\alpha(u_0,\ldots,u_{p-1})\neq0.\qedhere
\]
\end{proof}

\begin{samepage}
\begin{corollary}\label{cor:BC}
Hitchin's conjecture holds for complex simple Lie algebras of type
$B_\ell$ or $C_\ell$, with $\ell\geq2$, in every prime degree
$p=4r-1$ with $1\leq r\leq\ell$.
\end{corollary}

\begin{proof}
Let $\mathfrak g\subset\mathfrak{sl}_N$ be the standard inclusion, where
$N=2\ell+1$ in type $B_\ell$ and $N=2\ell$ in type $C_\ell$.
A principal $\mathfrak{sl}_2$ of $\mathfrak g$ remains principal in
$\mathfrak{sl}_N$ \cite[proof of Theorem 2.5]{BushekKumar2014}.
We choose a basis in which its triple is \eqref{eq:basis-action}
and put $k=2r-1$, so $p=2k+1$.
The algebra $\mathfrak g$ consists of operators
skew-adjoint for its defining symmetric or alternating bilinear form.
Since $k$ is odd, $e^k$ is again skew-adjoint and lies in
$\mathfrak g$. Its iterated commutators with $f$ also lie in
$\mathfrak g$, so $W_k\subset\mathfrak g$.

Since $N\geq2r=k+1$, Theorem~\ref{thm:main} implies that
$\omega_p$ is nonzero on $\Lambda^pW_k$.
Restriction on primitive classes is surjective for these inclusions
\cite[Theorem 3.5]{BushekKumar2014}.
By Proposition~\ref{prop:primitive-form}, every nonzero primitive
$p$-form on $\mathfrak g$ is therefore a nonzero multiple of the
restriction of $\omega_p$ and is nonzero on $\Lambda^pW_k$.
\end{proof}
\end{samepage}

\bigskip
\noindent
Boming Jia

\noindent
Yau Mathematical Sciences Center,\\
Jingzhai 301, Tsinghua University,\\
Beijing 100084, China

\noindent
Email \href{mailto:jiabm@tsinghua.edu.cn}{jiabm@tsinghua.edu.cn}

\end{document}